\documentclass[a4paper,12pt,leqno]{amsart}
\usepackage{amsmath,amsthm,amssymb,latexsym,epsfig,graphicx,subfigure}
\numberwithin{equation}{section} \numberwithin{figure}{section}

\newtheorem{thm}{Theorem}[section]
\theoremstyle{definition}
\newtheorem{lm}[thm]{Lemma}

\newtheorem{corollary}[thm]{Corollary}
\theoremstyle{definition}

\newtheorem{pr}[thm]{Proposition}
\theoremstyle{definition}
\newtheorem{df}[thm]{Definition}
\theoremstyle{definition}
\newtheorem{rem}[thm]{Remark}

\newcommand{\Rd}{\mathbb{R}^{d}}

\newcommand{\Rn}{\mathbb{R}^{n}}
\newcommand{\R}{\mathbb{R}}

\newcommand{\N}{\mathbb{N}}

\newcommand{\C}{\mathbb{C}}

\renewcommand{\S}{\mathcal{S}}

\newcommand{\ha}{\mathcal{H}}
\newcommand{\h}{\mathcal{H}}

\newcommand{\s}{\sigma}
\newcommand{\tmu}{\tilde{\mu}}
\newcommand{\om}{w}
\renewcommand{\O}{\Omega}

\def\sg{\sigma}
\newcommand{\stm}{\setminus}
\renewcommand{\o}{\Omega}

\DeclareMathOperator{\spv}{pv_{{\tt sy}}}
\DeclareMathOperator{\pv}{pv}

\newcommand{\e}{\varepsilon}

\newcommand{\cn}{C_{v|_n}}
\newcommand{\cnm}{C_{v|_{n+m_1 |w|}}}
\newcommand{\cnl}{C_{v|_{n+l |w|}}}
\newcommand{\ra}{\rightarrow}
\newcommand{\khii}{\text{\lower -.4ex\hbox{$\chi$}}}
\DeclareMathOperator{\spt}{spt} \DeclareMathOperator{\dist}{dist}
 
 \DeclareMathOperator{\diam}{diam}

\newcommand{\tsy}{T_{{\tt sy}}^\ast}
\renewcommand{\t}{\tau}
\newcommand{\f}{\Phi}
\renewcommand{\a}{\alpha}

\DeclareMathOperator{\re}{Re}

\begin{document}
\title{Homogeneous kernerls and self similar sets}
\author{Vasilis Chousionis}
\address{Department of Mathematics \\ University of Illinois \\ 1409
  West Green St. \\ Urbana, IL 61801}
\email{vchous@math.uiuc.edu}

\author{Mariusz Urba\'nski}
\address{Department of Mathematics \\University of North Texas\\ General Academics Building 435
  \\ 1155 Union Circle \#311430 \\ Denton, TX 76203-5017}
\email{urbanski@unt.edu}


\begin{abstract} 
We consider singular integrals associated to homogeneous kernels on self similar sets. Using ideas from Ergodic Theory we prove, among other things, that in Euclidean spaces the principal values of singular integrals associated to real analytic, homogeneous kernels fail to exist a.e. on self similar sets satisfying some separation conditions. Furthermore in general metric groups, using similar techniques, we generalize a criterion of $L^2$-unboundedness for singular integrals on self similar sets.
\end{abstract}

\maketitle
\section{Introduction}
The singular integrals, with respect to the Lebesgue measure in $\Rd$,
$$T(f)(x)=\int \frac{\Omega(x-y)}{|x-y|^{d}}f(y)dy$$
where $\Omega$ is a homogeneous function of degree $0$ have been studied extensively, see e.g. \cite{St}, being one of the standard topics in classical harmonic analysis. On the other hand if the singular integral is acting on general lower dimensional measures the situation is much more complicated even when one considers the simplest of kernels. The topic is deeply connected to geometric measure theory as it turns out that basic properties of singular integrals such as boundedness in $L^2$ and a.e. existence of principal values depend on the geometric structure of the underlying measure.

In a series of innovative works, see e.g. \cite{DS1} and \cite{DS2}, David and Semmes developed the theory of uniform rectifiability for the geometric study of singular integrals in $\Rd$ on Ahlfors-David regular (AD-regular, for short) measures, that is, measures $\mu$ satisfying
 \begin{equation*}
 \label{ad}
\frac{r^n}{C} \leq \mu(B(z,r)) \leq C r^n \text{ for }  z \in \operatorname{spt}\mu
\text{ and } 0<r<{\rm diam}(\operatorname{spt}(\mu)),
\end{equation*}
for some fixed constant $C$.
Roughly speaking, given a sufficiently nice kernel $k$ David and Semmes
intended to find geometric conditions that characterize the AD-regular measures $\mu$  for which the singular integrals $T_{k, \mu}$ are bounded in $L^2(\mu)$. To this end they introduced the novel concept of  uniform rectifiability which can be realized as rectifiability with some quantitative conditions. Recall that $n$-rectifiable  sets in $\Rd$ are contained, up to an $\ha^n$-negligible set, in a countable union  of $n$-dimensional Lipschitz graphs. Here $\ha^n$ stands for the $n$-dimensional Hausdorff measure.

When one assumes that the measure $\mu$ has sufficiently nice structure the situation is pretty well understood. David in \cite{d} proved that if $\mu$ is $n$-uniformly rectifiable any convolution kernel $k:\mathbb{R}^{d} \setminus \{0\}\rightarrow \mathbb{R}$
such that for all $x\in \mathbb{R}^{d}\setminus \{0\}$,
\begin{equation}
\label{ker}
k(-x)=-k(x)\; \ \text{ and }\ \;\left| \nabla^j k(x)\right| \leq c_j \left| x\right| ^{-n-j}, \ \text{ for}\ j=0,1,2,\dots
\end{equation} defines a bounded operator in $L^2(\mu)$. Later on, Tolsa in \cite{tun} relaxed the smoothness assumption up to $j=2$. Furthermore Verdera in \cite{Ve} proved that if the measure $\mu$ is $n$-rectifiable the principal values with respect to any odd kernel satisfying (\ref{ker}) exist $\mu$-a.e.

On the other hand, much less is known when one assumes $L^2$-boundedness or $\mu$-a.e existence of principal values and wishes to derive information about the geometric structure of $\mu$. David and Semmes in \cite{DS1} proved that the
$L^2(\mu)$-boundedness of all singular integrals in the class
described above forces the measure $\mu$ to be $n$-uniformly
rectifiable. Naturally one might ask what happens if in contrast to the previously mentioned result we only assume the boundedness of a single operator. Even for the $n$-dimensional Riesz kernels, $x /|x|^{n+1}, \ x \in \Rd \stm \{0\}$, the question, which is frequently referred to as the David-Semmes conjecture, remains partially unresolved.

In \cite{MMV}, Mattila, Melnikov and Verdera settled the David and Semmes question in the case of the Cauchy transform, that is for $n=1$. It is a remarkable fact that their proof depends crucially on a special subtle positivity property of the Cauchy kernel related to an old notion of curvature named after Menger. Recently, in a very deep work, Nazarov, Tolsa and Volberg \cite{ntov},  gave an affirmative answer to the David-Semmes conjecture in the case of the $(d-1)$-dimensional Riesz kernels. The conjecture remains open for $1<n<d-1$. 

Very little is known for other homogeneous kernels. In \cite{CMPT} the kernels $\re (z)^{2n-1} / |z|^{2n},$ $z \in \C, n \in \N,$ were considered and it was proved that the $L^2$-boundedness of the operators associated with any of these kernels implies rectifiability. Recently in \cite{CPr} the aforementioned result was extended to Euclidean spaces of arbitrary dimension. By now, these are the only known examples of convolution kernels not directly related to the Riesz kernels with this property. It is of interest that there exist some examples of homogeneous kernels in the plane whose boundedness does not imply rectifiability, see \cite{hu} and \cite{JN}. 

Mattila and Preiss proved in \cite{Mpr} that the $\mu$-a.e. existence of the principal values of the $n$-dimensional Riesz transforms implies $n$-rectifiability.
Huovinen in \cite{huothesis} considered the vectorial kernels $z^{2n-1} / |z|^{2n}, \, z \in \C, n \in \N,$ and proved that the $\mu$-a.e. existence of their principal values implies rectifiability. The same holds true for the kernels $\re (z)^{2n-1} / |z|^{2n}, \, z \in \C, n \in \N,$ considered in \cite{CMPT}.

It becomes clear that our knowledge restricts to a few particular examples of kernels.  Our goal in this paper is to prove, under certain restrictions, a general result. The idea is that given any sufficiently nice kernel it should behave badly on measures with sufficiently irregular geometric structure. In Theorem \ref{divpv} we prove that given any $s$-homogeneous real analytic kernel its principal values do not exist $\mathcal{H}^s$-a.e. in $C \subset \Rd$ if $C$ is a strongly separated, rotation free $s$-dimensional self-similar set. Furthermore in Theorem \ref{maxunb}, without even assuming strong separation for $C$, we prove that the corresponding maximal operator is unbounded in $L^\infty(\mathcal{H}^s \lfloor C)$. As a toy example the reader can have in mind the action of any kernel of the form $P(x)/|x|^{n+1}, x \in \R^2,$ where $P$ is an $n$-homogeneous polynomial, on the $1$-dimensional $4$-corners Cantor set in the plane.

In our proofs we make use of the fact that the zero set of any non-trivial real analytic function in $\Rd$ is contained in a countable union of $C^1$ manifolds of dimension at most $d-1$. Combined with Mattila's geometric rigidity theorem for self similar sets, see Theorem \ref{matss}, it allows us to prove that certain truncated integrals with respect to the real analytic kernel do not vanish on a set of positive $\mathcal{H}^s$-measure. A key and novel ingredient in our proof is the use of ideas and tools from Ergodic theory, especially suitable applications of Birkhoff's Ergodic Theorem. To our knowledge this is the first time that Birkhoff's Ergodic Theorem is being used in the context of singular integral operators.

Furthermore we use again Birkhoff's Ergodic Theorem in order to prove a criterion for unboundedness of homogeneous singular integrals on self similar sets of metric groups. This criterion was first obtained in \cite{CM}, with a quite different argument, under stronger separation conditions for the similarities. Motivation for the study of singular integrals in metric groups comes from the study of removable sets for Lipschitz $\mathcal{L}$-harmonic functions in Carnot groups, where $\mathcal{L}$ is the sub-Laplacian. It was proved in \cite{CM} that the critical dimension for such removable sets in the Heisenberg group $\mathbb{H}^n$ is $2n+1$ and the criterion for unboundedness was employed to prove the existence of removable self-similar sets with positive and finite $\mathcal{H}^{2n+1}$-measure.

The paper is organised as follows. In Section \ref{sec:notset} we lay down all the necessary notation and definitions regarding self similar iterated functions sytems and singular integrals in the general framework of complete metric groups. In Section \ref{sec:realanal} we consider singular integrals with respect to real analytic kernels in Euclidean spaces. In Section \ref{sec:unb} we prove a criterion for unboundedness of singular integrals on self similar sets of metric groups. Finally in Section \ref{sec:pert} we prove in Euclidean spaces the collection of homogeneous kernels that define unbounded operators on self similar sets is dense in the $C^r$ topology of homogeneous kernels.

\section{Notation and Setting}\label{sec:notset}
Let $(G,d)$ be a complete  metric group with the following properties:
\begin{enumerate}
\item The left translations $\tau_q:G \ra G$, 
$$ \tau_q(x)= q\cdot x, \ x \in G,$$
are isometries for all $q \in G.$
\item There exist dilations $\delta_r:G \ra G, \ r>0,$
which are continuous group homomorphisms for which,
\begin{enumerate}
\item $\delta_1=$ identity,
\item $d(\delta_r(x),\delta_r(y))=rd(x,y)$ for $x,y \in G,r>0$,
\item $\delta_{rs}=\delta_r \circ \delta_s$.
\end{enumerate}

It follows that for all $r>0$, $\delta_r$ is a group isomorphism with $\delta_r^{-1}=\delta_{\frac{1}{r}}$.
\end{enumerate}
The closed and open balls with respect to $d$ will 
be denoted by $B(p,r)$ and $U(p,r)$. By $\diam(E)$ we will denote the diameter of $E \subset G$ with respect to the metric $d$.

Let $E$ be a finite set called
in the sequel an alphabet. Without loss of generality we can assume that $E=\{1,\dots,N\}$ for some $N \in \N$. Let 
$$
\sg: E^\mathbb{N} \to E^\mathbb{N}
$$   
be the  shift map, i.e. cutting off the first coordinate. It is given by the
formula
$$  
\sg ( (w_n)^\infty_{n=1}  ) = ( (w_{n+1})^\infty_{n=1}  ).
$$
We also set
$$
E^*=\bigcup_{n=0}^\infty E^n.
$$
For every $\om \in E^*$, by $|\om|$  we mean the only integer
$n \geq 0$ such that $\om \in E^n$. We  call $|\om|$ the length of
$\om $. If $v \in E^*,\ \om \in E^\N$ and $n, m \geq 1$, we put
\begin{equation*}
\begin{split}
\om |_{n}&=\om_1\ldots \om_n\in E^n,\\
\om|_{m,\dots, m+n}&=\om_{m+1}\ldots \om_{m+n}\in E^n,\\
vw&=(v_1, \dots,v_{|v|},w_1,\dots)\in E^\N,\\
v^n&=(v_1,\dots,v_{|v|},\dots\dots,v_1,\dots,v_{|v|}) \in E^n,\\
v^\infty&=vv\dots \in E^\N.
\end{split}
\end{equation*} For every  $v \in E^*$, we denote the corresponding cylinder by 
$$
[v]:=\{\tau \in E^\mathbb{N}:\,\, \tau_{|_{|v|}}=v \},
$$ 
and if $A \subset E^\N$ we put
$$v \circ A=\{v\a: \a \in A\}.$$

Let $\mathcal{S}=\{S_i\}_{i \in E},$  be an iterated function system
(IFS) of similarities. This means that 
\begin{equation}
\label{simi}
d(S_i(x),S_i(y))=r_i d(x,y)
\end{equation}
with some $r_i \in (0,1)$ for all $i \in E$. The self-similar set $C$
is the invariant set 
with respect to $\mathcal{S}$, that is, the unique non-empty compact
set such that 
$$
C=\bigcup_{i \in E }S_i (C).
$$

We say that $\S$ satisfies the \textit{open set condition} (OSC) if
there exists some non-empty open set $O$ such that  
\begin{enumerate}
\item $S_i(O)\subset O$ for all $i\in E$,
\item $S_i(O) \cap S_j(O)=\emptyset$ for all $i \neq j \in E$.
\end{enumerate} 
If furthermore $O \cap C \neq \emptyset$ we say that $\S$ satisfies
the \textit{strong open set condition} (SOSC). Finally $\S$ is called
\textit{separated} if  
$$S_i (C) \cap S_j(C)=\emptyset\text{ for all }i,j
\in E \text{ with } i\ne j.$$  This equivalently means that
there exists some non-empty open set $O$ satisfying (i) and (ii) and
also (ii) with $O$ replaced by the closure of $O$.

Given any word $w=(w_1,\dots,w_n) \in E^\ast$ we  adopt the following
conventions: 
$$S_w:=S_{w_1}\circ \dots \circ S_{w_n}\ \text{and} \ C_w=S_w(C).$$
The periodic points of $\mathcal{S}$ are exactly those $x \in C$ such
that $S_w(x)=x$ for some $w \in E^\ast$. In this case 
$$\{x\}= \bigcap_{k=1}^\infty S_{w^k} (C).$$
We also define the coding map $\pi :E^{\N} \ra C$ by
$$\{\pi (w)\}= \cap_{n=1}^\infty S_{w |_n} (C).$$

We denote by $\mathcal{H}^s,s\geq 0,$ the $s$-dimensional Hausdorff
measure obtained from the metric $d$, i.e. for $E \subset G$ and
$\delta >0$, $\mathcal{H}^s (E)=\sup_{\delta>0} \mathcal{H}^s_\delta
(E)$, where 
$$
\mathcal{H}^s_\delta(E)=\inf \left\{\sum_i  \diam(E_i)^s: E \subset
  \bigcup_i E_i,\diam (E_i)<\delta \right\}.
$$

It follows by a general result of Hutchinson in \cite{Hut} that
whenever $\S$ is a finite set of similarities in $\Rd$ which satisfies
the OSC, 
$$0<\mathcal{H}^s(C)< \infty \ \text{ for } \ \sum_{i=1}^N r_i^s=1,$$ 
and the measure $\mathcal{H}^s \lfloor C$ is $s$-AD regular. Here
$\mathcal{H}^s \lfloor C$ stands for the restriction of the
$s$-dimensional Hausdorff measure on $C$. The real number $s$ is frequently called the similarity dimension of $\S$. In complete metric spaces the OSC does not always imply that the limit set has positive and finite $\ha^s$ measure. Nevertheless it holds true under some extra assumptions on the group $G$, see Section \ref{unb} for more details. We also remark that if $\S$ is separated it always generates a limit set with $0<\ha^s(C)<\infty$. 

From now on, unless
otherwise stated, we will denote $\mu =\mathcal{H}^s(C)^{-1}\mathcal{H}^s \lfloor C$. It follows, see e.g. \cite{Hut} that
$$\mu=\tmu \circ \pi^{-1}$$
where $\tmu$ is the canonical product measure in $E^\N$,
$$\tmu=\otimes_\N \left( \sum_{i \in E} r_i^s \epsilon_i\right)$$
where $\epsilon_i$ denotes the Dirac measure at $i \in E$. 

We will consider the following class of kernels.
\begin{df}
\label{hom} For $s>0$ the $s$-homogeneous kernels are of the form,
$$k (x,y)=\frac{\o(x^{-1}\cdot y)}{d(x,y)^s},\ x,y \in G \setminus \{(x,y):x=y\},$$
where $\o: G \ra \R$  is a not identically vanishing, continuous and
homogeneous function of degree zero, where $0$-homogeneity means that,  
$$\o(\delta_r(x))=\o(x)\ \text{for all} \ x\in G,r>0.$$
\end{df}

The truncated singular integral operators associated to $\mu$ and $k$ are defined for $f\in L^1(\mu)$ and $\e>0$ as,
$$T_\e(f)(y)=\int_{G \setminus B(x,\e)} k(x,y)f(y)d\mu (y),$$
and the maximal singular integral is defined as usual,
$$T^*(f)(x)=\sup_{\e>0}|T_\e(f)(x)|.$$

We say that the principal values with respect to  $k$ and $\mu$ exist for $x \in \spt \mu$ if the limit  
$$ \pv T (x):=\lim _{\e \ra 0} \int_{G \stm B(x,\e)} k(x,y) d \mu (y) $$ 
exists and it is finite.

We also introduce symbolic principal values and symbolic maximal singular integral operators.

\begin{df}
\label{sympv} Let $\mathcal{S}$ be  a set of separated similarities
and let $C$ be the corresponding $s$-dimensional self similar set. We
say that the \textit{symbolic principal values} with respect to a
kernel $k$ and $\mu=\mathcal{H}^s \lfloor C$ exist for $w \in E^\N$ if
the limit  
$$ \spv T (\pi (w)):=\lim _{k \ra \infty} \int_{C \stm C_{w|_k}} k(\pi
(w),y) d \mu (y) $$ 
exists and it is finite.
\end{df}

\begin{df}
\label{symmax}
Let $\mathcal{S}$ be a set of similarities satisfying the open set
condition  which generats a limit set $C$ such that $0<\ha^s(C)<\infty$. We define the  \textit{symbolic maximal singular
  operator} with respect to $k$ and $\mu=\mathcal{H}^s \lfloor C$ as 
$$T_{{\tt sy}}^\ast (f) (w)=\sup_{\substack{ m<n\\m,n \in \N }}
\left|\int_{C_{w|_m} \setminus C_{w|_n}}^\ast k(\pi(w),y)f(y) d \mu (y)
\right|$$ 
for $f \in L^1(\mu)$. Here we denote  $\int^\ast g d \mu
=\begin{cases} \int g d\mu &\mbox{if } \int g d \mu < \infty \\  
0 & \mbox{otherwise }  \end{cases} $.
\end{df}

\begin{rem}Notice that if $\mathcal{S}$ generates a separated self-similar set  $\int^\ast$ can be replaced by $\int$ in the above definition.
\end{rem}

\section{Real analytic kernels and self similar sets in $\Rd$}\label{sec:realanal}

In this section $(G,d)\equiv (\Rd, d_E)$, where $d_E$ is the Euclidean
metric, we focus our attention on the following class of kernels. 
\begin{df} 
\label{homra} For $s>0$, we say that $k \in \mathcal{G}_s$  if it is
of the form, 
$$
k (x,y)=\frac{\o(x-y)}{|x-y|^s},\ x \in \Rd \setminus \{0\},$$
where $\o: \Rd \stm{\{0\}}\ra \R$  is a non-trivial real analytic and
homogeneous function of degree zero.
\end{df}

\

\noindent Mattila in \cite{M1} proved the following geometric rigidity theorem for self similar sets.
 
\begin{thm}[\cite{M1}]
\label{matss} If $\mathcal{S}$ is a set of similarities satisfying the
open set condition and $C$ is their corresponding $s$-dimensional
limit set then either $C$ lies on an $n$-dimensional affine subspace
for some $n\le d$ or $\mathcal{H}^s(C \cap M)=0$ for any
$t$-dimensional $C^1$ submanifold $M$ where $t$ can be any number in
$(0,d)$. 
\end{thm}

If there exists some $n \in \N$ such that $C \subset V_n$ where $V_n$
is an $n$-dimensional affine subspace and $0<\mathcal{H}^n(C)<\infty$
we will call the self-similar set $C$ \textit{flat}. In this case we
can assume that the ambient space is $\R^n$ and it follows that
whenever $C$ is flat it has interior points and forms a local tiling,
see \cite{schiefosc}.  

We will be interested in non-flat self similar sets whose generating
similarities are separated and do not contain rotations. The latter
means that $\mathcal{S}=\{S_i\}_{i \in E},$  is a set of  similarities
of the form
\begin{equation}
\label{gsim}S_i=\t_{q_i}\circ \delta_{r_i}
\end{equation}
where $q_i \in \Rd,r_i \in (0,1)$ and $i=1,\dots,N$. Here as usual
$\delta_r(x)=rx,\, x\in \Rd, r>0$ and $\t_q(x)=q+x, \, q,x \in \Rd$,
denote respectively the dilations and translations in $\Rd$.

\begin{thm}
\label{divspv}
Let $\mathcal{S}$ be a set of separated, rotation-free similarities
which generates 
an $s$-dimensional self similar set $C$ and a
kernel $k \in \mathcal{G}_s$. Then  
\begin{enumerate}
\item the symbolic principal values with respect to $k$ and
  $\mu=\mathcal{H}^s \lfloor C$ do not exist $\tilde{\mu}$-a.e. in
  $E^{\N}$, 
\item the symbolic maximal operator $\tsy$ is unbounded in $L^\infty (\tmu)$.
\end{enumerate}
\end{thm}
\begin{proof} Let 
$$C_1:= S_1(C)=\pi([1]).$$
The function $f:(C \stm C_1)^c \ra \R$ defined by
$$f(x)=\int_{C \stm C_1} k(x,y)d \mu (y)$$
is real analytic in $(C \stm C_1)^c$. Furthermore $f$ is not
identically equal to zero. To see this fix $y_0 \in C\stm C_1$. Then
there exists $x_0 \in \partial B(y_0, 2\diam(C))$ such that
$\O(x_0-y_0):=\eta_0\neq 0$ and without loss of generality we can
assume that $\eta_0>0$. Hence, there exists some cylinder $[\alpha]$
such that  
$$
\O(x_0-y)>0 \  \text{ for all } \  y \in C_\alpha. 
$$
Notice also that for all $w \in E^\ast$, since $\S$ does not contain
rotations, we have
\begin{equation}
\label{homo}
\O (S_w(x)-S_w (y))=\O (x-y) \  \text{ for all }  \ x,y \in \Rd.
\end{equation}
Therefore,
\begin{equation*}
\begin{aligned}
0<\int_{C_\alpha \stm S_\alpha(C_1)}\frac{\O(x_0-y)}{|x_0-y|^s}d \mu(y)
&=\int_{C \stm C_1}\frac{\O(S_\a(S_\a^{-1}(x_0))-S_\a(z))}{|
  S_\a(S_\a^{-1}(x_0))-S_\a(z)|^s}|S_\a'|d \mu (z)\\ 
&=\int_{C \stm C_1}\frac{\O(S_\a^{-1}(x_0)-z)}{| S_\a^{-1}(x_0)-z|^s}d
\mu (z)\\
&=f(S_\a^{-1}(x_0)), 
\end{aligned}
\end{equation*}
after changing variables $y=S_\a (z)$. Hence
$f(S_\a^{-1}(x_0))>0$. Since $x_0 \notin C$ it follows that
$S_\a^{-1}(x_0)\notin C$ and $f:(C\setminus C_1)^c \ra \R$ is not
identically equal to zero. Let 
$$
Z_f=\{x\in(C \stm C_1)^c:f(x)=0\}.
$$
It follows by Lojasiewicz's Structure Theorem, see e.g. \cite{kr} ,
that $Z_f$ is a countable union of real analytic submanifolds whose
dimension does not exceed $d-1$. 
Since $\mathcal{S}$ is separated the limit $C$ is non-flat. This follows for example from \cite[Corollary 2.3]{schiefosc}. We therefore deduce from
Theorem~\ref{matss}  that  
$$
\mu(C_1 \cap Z_f)=0.
$$
Without loss of generality we can thus assume that there exists some
$x_1 \in C_1$ such that $f(x_1)>0$. Hence by the continuity of $f$
there exists some relatively open neighborhood 
$A_1 \subset C_1$ of $x_1$ (so $\mu(A_1)>0$) and  
\begin{equation}
\label{a1prop}
f(x)>\eta \  \text{ for all } \ x \in A_1 \ \text{ and some } \ \eta>0.
\end{equation}
The shift $\sigma:E^\N\to E^\N$ is a measure preserving and ergodic
transformation with 
respect to the measure $\tmu$. Since $\tmu(\pi^{-1}(A_1))=\mu(A_1)>0$,
Birkhoff's Ergodic Theorem yields 
$$
\lim_{n \ra \infty} \frac{1}{n} \sum_{k=0}^{n-1}
\chi_{\pi^{-1}(A_1)} (\sg^k (w)) =\tmu (\pi^{-1}(A_1)) >0
$$ 
for $\tmu$-a.e. $w \in E^\N$. Therefore if
$$
W=\{w \in E^{\N}:\text{ there exist infinitely many $k$'s such
  that $\sigma^k(w)\in \pi^{-1}(A_1)$}\},
$$ 
we see that $\tmu (W)=1$. Let $w \in W$. Let $x= \pi(w)$, and define 
$$G_w=\{k \in \N: \sg^k(w) \in \pi^{-1}(A_1) \}.$$
Now if $k \in G_w$ then $\sg^k(w) \in [1]$, that is $w_{k+1}=1$, and
after a change of variables we get 
\begin{equation}
\label{cov}
\begin{aligned}
\int_{C_{w|_{k}} \stm C_{w|_{k+1}}} k(x,y) d \mu (y)
&=\int_{C_{w|_{k}} \stm C_{w|_{k}1}} k(x,y) d \mu (y) \\
&= \int_{C \stm C_1} k (x, S_{w |_{k}}(y)) (r_{w_1} \dots r_{w_k})^s d \mu (y).
\end{aligned}
\end{equation}
Let $x' =\pi (\s^k(w))$. Then
$$
S_{w|_k}(x')
=S_{w|_k}(\pi (\s^k(w)))
=\pi(w|_k\s^k(w))
=\pi(w)
=x.
$$
Furthermore by the choice of $w$ and $k$ it follows that $x' \in A_1$. Hence by
(\ref{a1prop}) and (\ref{cov})
\begin{equation*}
\begin{split}
\int_{C_{w|_{k}} \stm C_{w|_{k+1}}} k(x,y) d \mu (y)&= \int_{C \stm
  C_1} k (S_{w |_{k}}(x'), S_{w |_{k}}(y)) (r_{w_1} \dots r_{w_k})^s d
\mu (y) \\ 
&= \int_{C \stm C_1} \frac{ \Omega (S_{w |_{k}}(x')-S_{w
    |_{k}}(y))}{d(S_{w |_{k}}(x'), S_{w |_{k}}(y))^s}(r_{w_1} \dots
r_{w_k})^s d \mu (y) \\ 
&= \int_{C\stm C_1} \frac{\Omega (x'-y)}{d(x',y)^s}d \mu (y) \\
&= f(x')> \eta .
\end{split}
\end{equation*}
Hence we have shown that for $\tmu$ a.e. $w \in E^\N$ there exists a sequence 
$G_w=\{k_i\}_{i  \in \N}$ such that 
\begin{equation}
\label{cauchy}
 \int_{C_{w|_{k_i}} \stm C_{w|_{k_i+1}}} k(\pi (w),y) d \mu (y)> \eta.
\end{equation}
Therefore for $\tmu$ a.e. $w \in E^\N$ the symbolic principal values
fail to exist. Hence we have proven (i). 

\vskip2mm \noindent For the proof of (ii) let $u \in E^\ast$ such that
$C_u \subset A_1$ where $A_1$ is as in (\ref{a1prop}). We now define a
sequence of maps $\{\f_k\}_{k \in \N}, \, \f_k:\Rd \ra \Rd$, by 
\begin{itemize}
\item $\f_0=Id$
\item $\f_k=S_{u^k}$ for all $k \geq 1$.
\end{itemize}
We denote
$$
x_1= \pi(u^\infty)\in C_u.
$$
The function $g:C \ra \R$  defined by 
$$g(y)=\sum_{k=0}^\infty \chi_{C \stm C_1}(\f_k^{-1}(y))$$
is either $1$ or $0$ therefore belongs to $L^\infty (\mu)$. Then for all $ m \in \N$,
\begin{equation}\label{g1}
\int_{C \stm \f_m (C_1)}  g(y)k_s(x_1,y)\, d \mu (y) =
\sum_{k=0}^{m}\int_{\f_k(C) \stm \f_k(C_1)}g(y)k(x_1,y) \, d \mu (y)
\end{equation}
Using the change of variables $y=\f_k (z)$ we have for all $k \in \N$
\begin{equation}
\label{g2}
\begin{split}
\int_{\f_k(C) \stm \f_k(C_1)}g(y)k(x_1,y)d \mu (y)&=\int_{\f_k(C)
  \stm \f_k(C_1)}\chi_{C \stm C_1}(\f_k^{-1}(y))k(x_1,y)d \mu (y)\\ 
&=\int_{C \stm C_1}\chi_{C \stm C_1}(z)k(x_1,\f_k (z))|\f_k'|^s d \mu (z)\\
&=\int_{C \stm C_1}\frac{\O(\f_k(\f_k^{-1}(x_1))-\f_k
  (z))}{|\f_k(\f_k^{-1}(x_1))-\f_k (z)|^s}|\f_k'|^s d \mu (z)\\
&=\int_{C \stm C_1} k(\f_k^{-1}(x_1),z)d \mu (z) \\
&=\int_{C \stm C_1} k(x_1,z)d \mu (z) \\
&=f(x_1)>\eta
\end{split}
\end{equation}
by (\ref{a1prop}) because $x_1\in C_u \subset A_1$.
Now let $M>0$ be an arbitrary  number and let $m \in \N$ such that $m
\eta > M$. Then by (\ref{g1}) and (\ref{g2}),
\begin{equation*}
\int_{C \stm \f_{m} (C_1)}  g(y)k(x_1,y)d \mu (y)>M.
\end{equation*}
By continuity of $k$ there thus exists some $m'>m$ such that
\begin{equation*}
\int_{C \stm \f_{m} (C_1)}  g(y)k(x,y)d \mu (y)>M
\end{equation*}
for all $x \in \f_{m'}(C)$. Therefore we have shown that there exists
a word $v \in E^\N$, which is just $v=u^\infty$, such that for all $M>0$
there exist $m_1, m_2 \in \N$, which depend on $M$, such that 
\begin{equation}
\label{bigm}
\int_{E^\N \stm [v|_{m_1}]}  g(\pi (\theta))k(\pi(w),\pi(\theta))d \tmu (\theta)>M
\end{equation}
for all $w \in [v|_{m_1+m_2}]$. Now let
$$
V=\{w \in E^{\N}:\text{ there exist  $n\ge 0$ such that $\sigma^n(w)\in
  [v|_{m_1+m_2}]$}\}.
$$ 
Applying Birkhoff's Ergodic Theorem as in (i) we obtain that
$\tmu(V)=\tmu(E^\N)$. Let $w \in V$ and let $n:=n(w)\ge 0$ be such that $\s^n(w)
\in [v|_{m_1+m_2}]$. Then using the change of variables $\tau= w|_n
\theta $ we get 
\begin{equation*}
\begin{split}
\int_{[w|_{n}]\stm [w|_{n+m_1}]} & g(\pi(\s^n(\tau)))k(\pi(w), \pi (\tau))d \tmu (\tau)=\\
&=\int_{ E^\N \stm [w|_{n,\dots, n+m_1}]}g(\pi(\s^n(w|_n \theta)))k(\pi(w), \pi (w|_n \theta))|S_{w|_n}'|^sd \tmu (\theta) \\
&=\int_{ E^\N \stm [v|_{m_1}]}g(\pi( \theta))k(S_{w|_n}(\pi(\s^n (w)), S_{w|_n}(\pi ( \theta)))|S_{w|_n}'|^sd \tmu (\theta)\\
&=\int_{ E^\N \stm [v|_{m_1}]}g(\pi( \theta))k(\pi(\s^n (w), \pi ( \theta))d \tmu (\theta)>M 
\end{split}
\end{equation*}
by (\ref{bigm}) because $\s^n(w) \in [v|_{m_1+m_2}]$. Hence we have
shown that for $\tmu$-a.e. $w \in E^\N$ there exists some $n(w)\ge 0$
such that $\tsy (g \circ \pi \circ \s^{n(w)})(w) > M$. Therefore there
exists some $n_0$ and some $B_{n_0} \subset E^\N$ with
$\tmu(B_{n_0})>0$ such that  
$$
\tsy (g \circ \pi \circ \s^{n_0})(w)>M
$$
for all $w \in B_{n_0}$. Thus $\|\tsy (g \circ \pi  \circ
\s^{n_0})\|_{L^\infty (\tmu)}>M$ while on the other hand $\| g \circ
\pi \circ  \s^{n_0}\|_{L^\infty (\tmu)} \leq 1$. Since $M$ was
arbitrary we have shown that $\tsy$ is not bounded in $L^\infty
(\tmu)$. 
\end{proof}
\begin{rem} If $\S$ satisfies the OSC then we can still define the symbolic principal values for $\tmu$-a.e. $w \in E^\N$ and the Theorem \ref{divspv} holds for all such sets of similarities $\S$ that generate non-flat limit sets.
\end{rem}
\

\noindent The following theorem follows immediately from (ii) of
Theorem \ref{divspv} and \cite[Lemma 2.4]{CM2}. 

\begin{thm}
\label{maxunb} 
Let $\mathcal{S}$ be a separated and rotation free set of similarities
which generates the 
$s$-dimensional self similar set $C$. If $k \in \mathcal{G}_s$, then the
maximal singular integral with respect to $k$ and $\mu=\mathcal{H}^s
\lfloor C$ is unbounded in $L^\infty(\mu)$. 
\end{thm}

\noindent We say that a set of similarities $\mathcal{S}$ is
\textit{strongly separated} if the corresponding self similar set $C$
satisfies 
$$
\min_{i \in E} \dist (S_i(C),C \stm S_i(C)) \geq \max_{i\in E}\diam(C_i).
$$
As another immediate corollary of Theorem \ref{divspv}, we have the
following theorem. 

\begin{thm}
\label{divpv} 
Let $\mathcal{S}$ be a strongly separated and rotation free set of
similarities which generates an 
$s$-dimensional self similar
set $C$. If $k \in \mathcal{G}_s$, then the principal values with respect
to $k$ and $\mu=\mathcal{H}^s \lfloor C$ do not exist $\mu$-a.e.. 
\end{thm}
\begin{proof} Since $\mathcal{S}$ is strongly separated, for all
  $v=(i_1,\dots,i_{|v|}) \in E^\ast$,  
\begin{equation}
\begin{split}
\label{ssdisj}
\dist (C_v,C_{v|_{|v|-1}} \stm C_v)
&=\dist(S_{v|_{|v|-1}}(S_{i_{|v|}}(C)),S_{v|_{|v|-1}}(C\stm S_{i_{|v|}}(C)))\\ 
&=r_{v_1} \dots r_{v_{|v|-1}} \dist(C, C \stm S_{i_{|v|}}(C))\\
& \geq  r_{v_1} \dots r_{v_{|v|-1}} \diam(S_{i_{|v|}}(C))\\
&=  r_{v_1} \dots r_{v_{|v|}} \diam(C)\\
&=\diam(C_v).
\end{split}
\end{equation}
Furthermore $C\stm C_v= \cup_{j=1}^{|v|} C_{v|_{j-1}} \stm C_{v|_j}$
and this union is disjoint. Therefore using (\ref{ssdisj}), we get
\begin{equation*}
\label{ssdisj2}
\begin{split}
\dist(C_v,C \stm C_v)&= \min_{j=1,\dots, |v|} \dist(C_v,C_{v|_{j-1}} \stm C_{v|_j}) \\
&\geq \min_{j=1,\dots, |v|} \dist(C_{v|_j},C_{v|_{j-1}} \stm C_{v|_j})\\
&\geq  \min_{j=1,\dots, |v|} \diam(C_{v|_j}) \\
&= \diam (C_v).
\end{split}
\end{equation*}
In particular this implies that for all $w \in E^\N$ and every $k \in \N$
\begin{equation*}
U(\pi(w), \diam(C_{w|_k})) \cap (C \stm C_{w|_k})=\emptyset 
\end{equation*}
and
\begin{equation*}
\mu(B(\pi(w), \diam(C_{w|_k})))=\mu( C_{w|_k}).
\end{equation*}
Hence, as in the proof of Theorem \ref{divspv}, by (\ref{cauchy}), for
$\tmu$-a.e. $w \in E^\N$ there exists a sequence $\{k_i\}_{i \in \N}$
such that 
\begin{equation*}
\begin{split}
\int_{B(\pi(w), \diam(C_{w|_{k_i}})) \stm B( \pi(w),
  \diam(C_{w|_{k_i+1}}))}& k(\pi (w),y) d \mu (y)\\
  &=\int_{C_{w|_{k_i}}
  \stm C_{w|_{k_i+1}}} k(\pi (w),y) d \mu (y)= \eta>0
\end{split}
\end{equation*}
and the principal values fail to exist for all such $\pi(w)$. 
\end{proof}

\section{The OSC and singular integrals in metric groups}\label{sec:unb}

In the context of complete metric spaces Schief proved
in \cite{s} that if $\S$ is a set of similarities as in (\ref{simi})
generating the limit set $C$  and  $\sum_{i=1}^N r_i^s=1$ then 
\begin{equation}
\label{mss}
\h^s (C) >0 \quad \Longrightarrow \quad \text{SOSC}.
\end{equation} If furthermore the space is doubling, that is there
exists some  $N \in \N$ such that for all $x$ and all $r>0$ there
exist $\{x_i\}_{i=1}^N$ such that 
$$B(x,r) \subset \cup_{i=1}^N B(x_i,r/2),$$
Balogh and Rohner proved in \cite{br} that
\begin{equation}
\label{br}
\text{OSC}\quad\Longleftrightarrow \quad 0<\h^s(C)< \infty.
\end{equation}

We remark that if $(G,d)$ is a locally compact metric group the
left Haar  measure $\lambda$ on $G$ is doubling that is, there exist
some constant $C$ such that for all $x \in G$ and $r>0$, 
$$\lambda (B(x,2r))\leq C \lambda (B(x,r)),$$
see e.g. \cite[Proposition 2.14]{ms}. By an observation of Coifman and
Weiss in \cite{cw} the existence of a doubling measure on $G$ forces
the metric space to be doubling. Therefore whenever $\S$ is a set of
similarities in a locally compact metric group, (\ref{mss}) and
(\ref{br}) imply that 
$$
\text{OSC}\quad \Longleftrightarrow \quad \text{SOSC}.
$$

From now on $(G,d)$ will be a doubling, complete metric group with dilations as in Section \ref{sec:notset} and $\mathcal{S}=\{S_1,\dots,S_N\},N \geq 2,$ will  be an
iterated function system of similarities of the 
form
\begin{equation}
\label{gsim2}S_i=\t_{q_i}\circ \delta_{r_i}
\end{equation}
where $q_i \in G,r_i \in (0,1)$ and $i=1,\dots,N$. 

\begin{thm}
\label{unb} 
Let $\S$ be an IFS as in (\ref{gsim2}) which satisfies the OSC and
generates an $s$-dimensional self similar set $C$. Let $k$ an
$s$-homogeneous kernel. If there exists a periodic point $x_w, \ w \in
E^\ast$, such that $x_w \in O$ for some open set $O$ for which $\S$
satisfies the SOSC, and 
\begin{equation*}
\int_{C\stm C_w} k(x_w,y)d \mu (y) \neq 0
\end{equation*}
then $T^\ast (1) (x)=\infty$ for $\mu$-a.e. $x$.
\end{thm}

\noindent The essential step in the proof of Theorem \ref{unb} is the
following proposition. 

\begin{pr}
\label{symunb}
Let $\S$ be and IFS as in (\ref{gsim2}) which satisfies the OSC and
generates an $s$-dimensional self similar set $C$. Let $k$ an
$s$-homogeneous kernel. If there exists a periodic point $x_w, \ w \in
E^\ast$, such that 
\begin{equation}
\label{int}
\int_{C\stm C_w} k(x_w,y)d \mu (y) \neq 0
\end{equation}
then $\tsy (1) (v)=\infty$ for $\tmu$-a.e. $v \in E^\N$.
\end{pr}
\begin{proof}For simplicity we denote $x=x_w$. Without loss of
  generality we can assume that 
$$\int _{C \setminus C_w} k(x,y)d \mu (y) = \eta >0.$$
Notice that the homogeneity of $\o$ implies that for all $v \in E^\ast$,
\begin{equation}
\label{homome}
\o (S_v (x)^{-1} \cdot S_v (y))=\o(\delta_{r_{i_1}\dots
  r_{i_{|v|}}}(x^{-1}\cdot y))=\o (x^{-1}\cdot y). 
\end{equation}
Hence
\begin{equation}
\label{homome1}
k(S_v(x),S_v(y))=k(x,y)(r_{v_1}\dots v_{|v|})^s. 
\end{equation}
Therefore for all $k \in \N$, after changing variables $y=S_{w^k}(z)$
and recalling that $S_{w^k}(x)=x$, 
\begin{equation*}
\begin{split}
\int_{C_{w^k}\stm C_{w^{k+1}}} k(x,y)d \mu y&=\int_{C_{w^k}\stm C_{w^{k+1}}} \frac{\o(x^{-1}\cdot y)}{d(x,y)^s} d\mu (y)\\
&=\int _{C \stm C_w} \frac{\o(x^{-1}\cdot S_{w^k}(z))}{d(x,S_{w^k}(z))^s} (r_{w_1}\dots r_{w_{|w|}})^{ks} d \mu(z)\\
&=\int _{C \stm C_w} \frac{\o(S_{w^k} (S_{w^k}^{-1}(x))^{-1}\cdot S_{w^k}(z))}{d(S_{w^k} (S_{w^k}^{-1}(x)),S_{w^k}(z))^s} (r_{w_1}\dots r_{w_{|w|}})^{ks} d \mu (z) \\
&=\int _{C \stm C_w} \frac{\o(S_{w^k}^{-1}(x)^{-1}\cdot z)}{d(S_{w^k}^{-1}(x),z)^s} d \mu (z)\\
&=\int _{C \stm C_w}\frac{\o(x^{-1}\cdot z)}{d(x,z)^s} d\mu (z) \\
&= \eta.
\end{split}
\end{equation*}
Let $M$ be an arbitrary positive number and choose $m_1 \in \N$ such
that $m_1\eta >M$. Then 
$$\int_{C\stm C_{w^{m_1}}} k(x,y)d \mu
(y)=\sum_{i=0}^{m-1}\int_{C_{w^{i}}\stm C_{w^{i+1}}} k(x,y)d \mu
(y)>M.$$ 
Therefore by the continuity of $k$ away from the diagonal there exist
$m_2>m_1,$ such that 
\begin{equation}
\label{mest}
\int_{C\stm C_{w^{m_1}}} k(\pi(\tau),y)d \mu (y) >M \  \text{ for all
}\  \tau \in [w^{m_2}].
\end{equation}
Let 
$$A=\{v \in E^\N: \ \text{there exists} \ n \in \N \ \text{such that}
\ \s^n(v)\in [w^{m_2}]\}$$ 
Then as in the proof of Theorem \ref{divspv}, Birkhoff's Ergodic
Theorem implies that $\tmu (A)=1$. For $v \in A$ set 
$$G_v=\{n \in \N:\ \s^n(v) \in [w^{m_2}] \}.$$  
Then $G_v \neq \emptyset$ and for $n \in G_v$ we have
\begin{equation*}
\begin{split}
\int_{[v|_n] \stm [v|_{n+m_1 |w|}]}^\ast & k(\pi(v), \pi(\tau)) d \tmu(\t)=\int^\ast_{v|_n \circ (E^\N \stm [v|_{n+1,\dots,n+m_1 |w|}])} k(\pi(v), \pi(\tau)) d \tmu(\t)\\
&=\int^\ast_{v|_n \circ (E^\N \stm [w^{m_1}])} k(\pi(v), \pi(\tau)) d \tmu(\t).
\end{split}
\end{equation*}
The last equality follows because $\s^n(v) \in [w^{m_2}]$ and $m_2>m_1$. Hence after a change of variables $\tau = v|_n \theta$
\begin{equation}
\begin{split}
\label{maxest}
\int_{[v|_n] \stm [v|_{n+m_1 |w|}]}^\ast & k(\pi(v), \pi(\tau)) d \tmu(\t)=\int^\ast_{E^\N \stm [w^{m_1}]} k(\pi(v), \pi(v|_n \theta)) \ (r_{v_1}\dots r_{v_n})^{s} d \tmu(\theta)\\
&=\int^\ast_{E^\N \stm [w^{m_1}]} k( S_{v|_n}(\pi(\s^n(v))),S_{v|_n}( \pi( \theta))) \ (r_{v_1}\dots r_{v_n})^{s} d \tmu(\theta) \\
&=\int_{E^\N \stm [w^{m_1}]} k( \pi(\s^n(v)), \pi( \theta)) \, d \tmu(\theta)\\
&=\int_{C \stm C_{w^{m_1}}} k(\pi(\s^n(v)),y)\, d \mu (y)>M
\end{split}
\end{equation}
because $\s^n(v) \in [w^{m_2}]$. Since $M$ was arbitrary we deduce that for $\tmu$-a.e. $v \in E^{N}$
$$\tsy(1)(v)=\infty.$$
\end{proof}

\

\begin{proof}[Proof of Theorem \ref{unb}]
Let $X= \overline{O}$. Then there exists $n_0:=n_0(w) \in \N$ and
$c_0:=c_0(w)>0$ such that for all $m \in \N$ 
\begin{equation}
\label{seper}
\dist (S_{w^{m+n_0}}(X), \partial S_{w^m}(X)) \geq c_0 \diam (S_{w^m}(X))
\end{equation}
To see this notice that as $x_w \in O$ there exists $r=r_w>0$ such
that $\dist(B(x_w,r),\partial X)>0 $. Therefore there exists some
$n_0:=n_0(w) \in \N$ such that  
$$\dist(S_{w^{n_0}}(X),\partial X):=d_0>0.$$
Let $c_0:=d_0\diam^{-1}(X)$. Then for all $m \in \N$
\begin{equation*}
\begin{split}
\dist &(S_{w^m}(S_{w^{n_0}}(X)), \partial
S_{w^m}(X))=\dist(S_{w^m}(S_{w^{n_0}}(X)), S_{w^m}(\partial (X)))\\ 
&\quad=(r_{w_1}\dots r_{w_{|w|}})^{m} \dist (S_{w^{n_0}}(X), \partial
(X))\geq (r_{w_1}\dots r_{w_{|w|}})^{m} c_0 \diam (X)\\ 
&\quad=c_0 \diam (S_{w^m}(X)),
\end{split}
\end{equation*}
and (\ref{seper}) follows.
Let $M>0$ be arbitrary and let $m_2, m_1 \in \N$ be as in the proof of
Proposition~\ref{symunb}, that is they satisfy (\ref{mest}). Let  
$$
A'=\{v \in E^\N: \ \text{there exists} \ n \in \N \ \text{such that}
\ \s^n(v)\in [w^{m_2+n_0}]\}
$$ 
Then Birkhoff's Ergodic Theorem implies that $\tmu (A')=1$. For $v \in A'$ set
$$G'_v=\{n \in \N:\ \s^n(v) \in [w^{m_2+n_0}] \}.$$
Then  exactly as in (\ref{maxest}) we get that for all $ v \in A'$ and
for all $n \in G' (v)$ 
\begin{equation}
\label{mest2}
\int_{C_{v|_n} \stm C_{v|_{n+m_1|w|}}}k(\pi(v),y)d \mu (y)=\int_{[v|_n]
  \stm [v|_{n+m_1 |w|}]}  k(\pi(v), \pi(\tau)) d \tmu(\t)>M. 
\end{equation}
\begin{lm}
\label{comp} If $v \in A'$ and $n \in G'_v$, then there exists a
constant $c_1:=c(w)$ such that 
$$
\left|\int_{C_{v|_n} \stm C_{v|_{n+m_1|w|}}}k(\pi(v),y)d \mu
  (y)-\int_{B_1 \stm B_2} k(\pi(v),y) d \mu (y) \right| \leq c_1,
$$
where $B_1=B(\pi(v),2 \diam (C_{v|_n}))$ and $B_2=B(\pi(v),2 \diam
(C_{v|_{n+m_1|w|}}))$ 
\end{lm}
\begin{proof} We have 
$$\cn \stm \cnm=(\cn \stm B_2) \cup ((B_2 \stm \cnm) \cap \cn) $$
and
$$
B_1 \stm B_2=(B_1 \stm (\cn \cup B_2))\cup (\cn \stm B_2),
$$
where the unions are disjoint. Hence
\begin{equation*}
\begin{split}
\int_{\cn \stm \cnm}&k(\pi(v),y)d \mu (y) + \int_{B_1 \stm (\cn \cup
  B_2)}k(\pi(v),y)d \mu (y)\\ 
&= \int_{B_1 \stm B_2}k(\pi(v),y)d \mu (y) + \int_{(B_2 \stm \cnm)
  \cap \cn}k(\pi(v),y)d \mu (y) 
\end{split}
\end{equation*}
and so,
\begin{equation}
\begin{split}
\label{compoper}
&\left| \int_{C_{v|_n} \stm C_{v|_{n+m_1|w|}}}k(\pi(v),y)d \mu
  (y)-\int_{B_1 \stm B_2} k(\pi(v),y) d \mu (y) \right|\\ 
 & \quad \quad \leq\|\O\|_{L^\infty (\mu)}\left( \int_{B_1 \stm \cn}
   d(\pi(v),y)^{-s} d \mu (y) + \int_{B_2 \stm \cnm}d(\pi(v),y)^{-s} d
   \mu  (y)\right)\\ 
 &\quad \quad:= \|\O\|_{L^\infty (\mu)}(I_1+I_2).
\end{split}
\end{equation}
For all integers $0 \leq l \leq m_1$ we have
$$
C\stm \cnl \subset  \bigcup_{\substack{|\t|=n+l|w| \\ \t \neq v|_{n+l
      |w|}}}S_{\t}(C) \subset  \bigcup_{\substack{|\t|=n+l|w| \\ \t \neq
    v|_{n+l |w|}}}S_{\t}(X) \subset \overline{X \stm S_{v|_{n+l|w|}}(X)}.
$$
Hence
\begin{equation}
\label{seper2}
\begin{split}
\dist(\pi(v),C \stm \cnl)  & \geq  \dist(\pi (v), X \stm S_{v|_{n+l |w|}}(X))\\
&= \dist(\pi (v), \partial S_{v|_{n+l |w|}}(X))\\
&=\dist(S_{v|_n}(\pi(\s^n(v))), S_{v|_n}(\partial S_{v|_{n+1, \dots,n+l |w| }}(X)))\\
&= (r_{v_1}\dots r_{v_{n}})\dist(\pi(\s^n(v)), \partial S_{w^l}(X)) \\
&=(r_{v_1}\dots r_{v_{n}})\dist(\pi(\s^n(v)), X \stm S_{w^l}(X)) \\
&\geq c_0 (r_{v_1}\dots r_{v_{n}}) \diam (S_{w^l}(C))\\
&=c_0 \diam (S_{v|_{n+l|w|}}(C)),
\end{split}
\end{equation}
where we used (\ref{seper}) and the fact that $\s^n(v) \in
[w^{m_2+n_0}] \subset [w^{l+n_0}]$ as  $m_2>m_1 \geq l$. 
Applying (\ref{seper2}) for $l=0$ and for $l=m_1$ we obtain that 
$I_1+I_2 \leq c_2$ with some constant $c_2$ and the lemma follows by
(\ref{compoper}) with $c_1:=c_2\|\O\|_{L^\infty (\mu)}$. 
\end{proof}

\

\noindent Since $M$ was an arbitrarily large number Theorem \ref{unb}
follows by (\ref{mest2}) and Lemma \ref{comp}.
\end{proof}

\section{$C^r$-perturbations of Kernels}\label{sec:pert}

\noindent Fix $r\in\{0,1,2,\ldots,\infty\}$. Let $\o: \Rd
\stm{\{0\}}\ra \R$ be a non-trivial $C^r$ and homogeneous function of
degree zero. Equivalently, and this will be more convenient
throughout this section, $\o$ can be treated as a $C^r$ function $\hat
\o$ from the unit sphere $S^{d-1}$ to $\R$. Let $\mathcal{F}^{r}$ be
collection of all such $\o$s.  
Let $\mathcal{S}=\{S_i\}_{i\in E}$ be a separated IFS consisting of
similarities. Let $w\in E^*$ be a finite word and let $\xi_w$ be the
only fixed point of $S_w$. As in the previous sections let $s$ be the
similarity dimension of the limit set $C$. We shall prove the following.

\

\begin{pr} \label{pertpro}
Let $\mathcal{S}=\{S_i\}_{i\in E}$ be a separated IFS consisting of
similarities which generates an $s$-dimensional limit set $C$. For every $r\in\{0,1,2,\ldots,\infty\}$ and every 
finite word $w\in E^*$ the subcollection $\mathcal{F}^r(w)$ of $\mathcal{F}^r$
consisting of all elements $\o$ such that 
\begin{equation}\label{pert1}
\int_{C\setminus C_w}\frac{\o(x-\xi_w)}{|x-\xi_w|^s}\,d\mu(x)\ne 0
\end{equation}
is open in $C^0$ topology, dense in $C^r$ topology if $r$ is finite, and in $C^k$ topology for every finite $k$ if $r=\infty$. In consequence $\mathcal{F}^r(w)$ is open and dense respectively in $C^r$ topology or all $C^k$ topologies.
\end{pr}
\begin{proof}
The openess statement is obvious. So, let us deal with denseness. Fix
$\o\in\mathcal{F}^r$. Our task is to find elements of $\mathcal{F}^r$
arbitrarily close to $\o$ in $C^r$ topology for which the integral in
\eqref{pert1} does not vanish. Fix $z\in C\setminus C_w$. Define the function
$H:\R^d\setminus\{0\}\to S^{d-1}$ by
$$
H(x)=\frac{x-\xi_w}{|x-\xi_w|}.
$$
Consider $U$, an open ball contained in $S^{d-1}$ such that $H(z)\in U$. Given
$\e>0$, by the $C^\infty$ version of Urysohn's Lemma there exists a
$C^\infty$ function $g:S^{d-1}\to\R$ such that  
$$
g|_{S^{d-1}\setminus U}=0,  \    \  g|_U>0,
$$
and all the derivatives of $g$ from  order $0$ up to $r$ are less than
$\epsilon$. 
Define $\o^*:\R^d\setminus\{0\}\to\R$ by declaring that $\hat\o^*=\hat\o+g$. Then 
$$
\hat\o^*\ge \hat\o  \  \text{ on }  \  S^{d-1}  \  \text{ and }  \
\hat\o^*> \hat\o  \  \text{ on }  \  U. 
$$
If 
$$
\int_{C\setminus C_w}\frac{\o(x-\xi_w)}{|x-\xi_w|^s}\,d\mu(x)\ne 0,
$$
we are done; there is nothing to do. Otherwise, $\o^*$ is $\epsilon$-close to $\o$ in $C^r$ and, as $\mu(H^{-1}(U)\setminus C_w)>0$, we get
$$
\int_{C\setminus C_w}\frac{\o^*(x-\xi_w)}{|x-\xi_w|^s}\,d\mu(x)
>\int_{C\setminus C_w}\frac{\o(x-\xi_w)}{|x-\xi_w|^s}\,d\mu(x)
>0.
$$
The proof is complete.
\end{proof}

As an immediate corollary of Proposition \ref{pertpro} and Theorem \ref{unb} we have the following. Let $\mathcal{K}^r$ be the collection of $s$-homogeneous kernels $k$ which are $C^r$-away from the origin and let $\mathcal{U}^r$ be the subcollection of $\mathcal{K}^r$ consisting of all kernels $k$ such that $T^*_k(1)(x)=\infty$ for $\ha^{s}$-a.e.

\begin{corollary}
Let $\mathcal{S}=\{S_i\}_{i\in E}$ be a separated IFS consisting of
similarities which generates an $s$-dimensional limit set $C$. For every $r\in\{0,1,2,\ldots,\infty\}$ let $\mathcal{K}^r$ be the collection of $s$-homogeneous kernels $k$ which are $C^r$-away from the origin and let $\mathcal{U}^r$ be the subcollection of $\mathcal{K}^r$ consisting of all kernels $k$ such that $T^*_k(1)(x)=\infty$ for $\ha^{s}$-a.e $x \in C$. Then $\mathcal{U}^r$ is open and dense in the $C^r$-topology of $\mathcal{K}^r$.
\end{corollary}


\begin{thebibliography}{CMPT}


\bibitem[BR]{br} Z. Balogh, H. Rohner, \emph{ Self-similar sets in doubling spaces}, Illinois J. Math. 51 (2007), no. 4, 1275--1297



\bibitem[CMPT]{CMPT}  V. Chousionis, J. Mateu, L. Prat and X. Tolsa,
\emph{ Calder\'on-Zygmund kernels and rectifiability in the plane},  Adv. Math. 231:1 (2012), 535--568.




\bibitem[CM]{CM} V. Chousionis and P. Mattila, \emph{Singular integrals on Ahlfors-David subsets of the Heisenberg group}, J. Geom. Anal. 21 (2011), no. 1 , 56--77.

\bibitem[CM]{CM2} V. Chousionis and P. Mattila, \emph{Singular integrals on self-similar sets and removability for Lipschitz harmonic functions in Heisenberg groups}, to appear in J. Reine Angew. Math.

\bibitem[CP]{CPr}  V. Chousionis and L. Prat,
\emph{Some Calder\'on-Zygmund kernels and their relation to rectifiability and Wolff capacities},  preprint.

\bibitem[CW]{cw} R. Coifman and G. Weiss \emph{Analyse harmonique non-commutative sur certains espaces homog\'enes}, (French) Lecture Notes in Mathematics, Vol. 242. Springer-Verlag, Berlin-New York, 1971. 

\bibitem[D1]{d} G. David,\emph{Unrectifiable 1-sets have vanishing analytic capacity}, Rev. Math. Iberoam. 14 (1998) 269--479.



\bibitem[DS1]{DS1} G. David and S. Semmes, \emph{Singular Integrals and rectifiable sets in $\mathbb {R}
^n$: Au-del\'{a} des graphes lipschitziens}, Asterisque
193, Soci\'{e}t\'{e} Math\'{e}matique de France (1991).



\bibitem[DS2]{DS2} G. David and S. Semmes, \emph{Analysis of and on Uniformly Rectifiable Sets}, Mathematical Surveys and Monographs, 38. American Mathematical Society, Providence, RI, (1993).
















\bibitem[Hut]{Hut} J. Hutchinson, \emph{Fractals and self-similarity}, Indiana Univ. Math. J. 30 (1981), no. 5, 713--747.

\bibitem[H1]{huothesis} P. Huovinen, \emph{Singular integrals and rectifiability of measures in the plane}, Ann. Acad. Sci. Fenn. Math. Diss. No. 109 (1997)

\bibitem[H2]{hu} P. Huovinen,  \emph{A nicely behaved singular integral on a purely unrectifiable set}, Proc. Amer. Math. Soc. 129 (2001), no. 11, 3345--3351.




\bibitem[JN]{JN} B. Jaye and F. Nazarov, \emph{Three revolutions in the kernel are worse than one}, preprint.




\bibitem[K]{kr} S. Krantz and H. Parks, \emph{A primer of real analytic functions}, Second edition. Birkhäuser Advanced Texts: Basler Lehrbücher. [Birkhäuser Advanced Texts: Basel Textbooks] Birkhäuser Boston, Inc., Boston, MA, 2002.

\bibitem[NToV]{ntov} F. Nazarov, X. Tolsa and A. Volberg,  {\em On the uniform rectifiability of AD-regular measures with bounded Riesz transform operator: the case of codimension 1.} submitted (2012)



\bibitem[M1]{M1} P. Mattila, \emph{On the structure of self-similar fractals},
Ann. Acad. Sci. Fenn. Ser. A I Math. 7 (1982), 189--195.

\bibitem[M2]{M} P. Mattila, \emph{Geometry of Sets and Measures in Euclidean Spaces},
Cambridge University Press, (1995).



\bibitem[M3]{ms} P. Mattila, \emph{Measures with unique tangent measures in metric groups}, Math.
Scand. 97 (2005), 298--398.


\bibitem[MMV]{MMV} P. Mattila, M. S. Melnikov and J. Verdera,\emph{The Cauchy integral, analytic capacity, and uniform rectifiability}, Ann. of Math. (2)  144  (1996),  no. 1, 127--136.



\bibitem[MP]{Mpr}  P. Mattila and  D. Preiss, \emph{Rectifiable measures in $\Rn$ and existence of principalvalues for singular integrals}, J. London Math. Soc., 52 (1995),482-496.





\bibitem[S1]{schiefosc} A. Schief, \emph{Separation properties for self-similar sets }, Proc. Amer. Math. Soc. 122 (1994), no. 1, 111--115.
Proc. Amer. Math. Soc. .

\bibitem[S2]{s} A. Schief, \emph{Self-similar sets in complete metric spaces}, 
Proc. Amer. Math. Soc. 124 (1996), 481-490.



\bibitem[St]{St} E. M. Stein, \emph{Singular integrals and differentiability properties of functions}, -Princeton University Press, Princeton New Jersey, (1976).










\bibitem[T]{tun} X. Tolsa \emph{Uniform rectifiability, Calder\'on-Zygmund operators with odd kernel, and quasiorthogonality}, Proc. London Math. Soc. 98(2) (2009), 393--426.

\bibitem[V]{Ve} J. Verdera,
\emph{A weak type inequality for Cauchy transforms of finite
measures},  Publ. Mat. 36 (1992), no. 2B, 10291034.


\end{thebibliography}
\end{document}